\documentclass[12pt,twoside,reqno]{amsart}
\usepackage[colorlinks=true,citecolor=blue]{hyperref}
\usepackage{mathptmx, amsmath, amssymb, amsfonts, amsthm, mathptmx, enumerate, color}
\usepackage{graphicx}
\usepackage{epstopdf}
\usepackage{xcolor}
\allowdisplaybreaks
\usepackage{enumitem}
\usepackage{csquotes}
\MakeOuterQuote{"}

\newtheorem{theorem}{Theorem}[section]

\newtheorem{lemma}{Lemma}[section]
\newtheorem{proposition}{Proposition}[section]
\theoremstyle{definition}
\newtheorem{definition}{Definition}[section]
\newtheorem{example}{Example}[section]
\newtheorem{remark}{Remark}[section]

\numberwithin{equation}{section}

\newtheorem*{assumption*}{\assumptionnumber}
\providecommand{\assumptionnumber}{}
\makeatletter
\newenvironment{assumption}[2]
{%
	\renewcommand{\assumptionnumber}{Assumption #1#2}%
	\begin{assumption*}%
		\protected@edef\@currentlabel{#1#2}%
	}
	{%
	\end{assumption*}
}
\makeatother
\begin{document}
\setcounter{page}{1}

\vspace*{1.0cm}
\title[Prox-Regular Perturbed Sweeping Processes]
{On the Solution Existence for Prox-Regular Perturbed Sweeping Processes}
\author[Nguyen Khoa Son, Nguyen Nang Thieu, Nguyen Dong Yen]{Nguyen Khoa Son$^{1}$, Nguyen Nang Thieu$^{2,1}$, Nguyen Dong Yen$^{1,*}$}
\maketitle
\vspace*{-0.6cm}

\begin{center}
{\footnotesize {\it

$^1$Institute of Mathematics, Vietnam Academy of Science and Technology, Hanoi, Vietnam\\
$^2$XLIM UMR-CNRS 7252, Universit\'e de Limoges, Limoges, France

}}\end{center}

\vskip 4mm {\small\noindent {\bf Abstract.}
In the setting adopted by Edmond and Thibault [Mathematical Programming 104 (2005), 347--373], we study a class of perturbed sweeping processes. Under suitable assumptions, we obtain two solution existence theorems for perturbed sweeping processes with the constraint sets being prox-regular sublevel sets. The results are applied to analyzing the behavior of some concrete mechanical sweeping processes, which appear for the first time in this paper.

 }

\renewcommand{\thefootnote}{}
\footnotetext{$^*$Corresponding author.
\par
E-mail addresses: nkson@vast.ac.vn (Nguyen Khoa Son), nguyennangthieu@gmail.com (Nguyen Nang Thieu), ndyen@math.ac.vn (Nguyen Dong Yen).
\par
Received xxxxx; Accepted xxxx.}

\section{Introduction}

Let $T>0$ be a real number and let $C(t)$, $t\in [0,T]$, be nonempty closed subsets of a Hilbert space ${\mathcal H}$. For any fixed $x_0\in C(0)$, the differential inclusion
\begin{equation}\label{sweepingprocess}
\begin{cases}
-\dot{x}(t)\in {\mathcal{N}}_{C(t)}(x(t)) \quad \text{a.e.}\ t\in [0,T],\\
x(0)=x_0,
\end{cases}
\end{equation}
where ${\mathcal{N}}_{\Omega}(z)$ denotes the Clarke normal cone~\cite[p. 51]{Clarke} to a closed set $\Omega$ at $z$, is called a \textit{sweeping process}. An absolutely continuous function $x(\cdot):[0,T]\to {\mathcal H}$ which satisfies the two conditions in~\eqref{sweepingprocess} is said to be a \textit{solution} of the sweeping process. It is worthy to stress that any absolutely continuous function $x(\cdot):[0,T]\to {\mathcal H}$ is Fr\'echet differentiable almost everywhere on $[0,T]$ with respect to the Lebesgue measure (see Proposition~\ref{differentiablity} below). If $C(t)$ is convex, then the Clarke normal cone coincides with the normal cone in the sense of convex analysis~\cite[Proposition 2.4.4, p. 52]{IT}. 

The model~\eqref{sweepingprocess} under the assumption that $C(t)$ is \textit{convex} for each $t \in [0,T]$ was introduced by Moreau in~\cite{Moreau4}, where some fundamental results on solution existence and uniqueness were obtained. In~\cite{Moreau5}, he has studied the continuity of the solutions when the convex-valued mapping $C:[0,T]\rightrightarrows{\mathcal H}$ undergoes small perturbations.

In many subsequent papers, assumptions on the convexity of $C(t)$ have been relaxed. For examples, Colombo and Goncharov~\cite{ColomboGoncharov} obtained a solution existence and uniqueness theorem for the sweeping process~\eqref{sweepingprocess} under the hypothesis that the sets $C(t)$ are weakly closed and $\varphi$-convex. Later, in a more general setting, Bounkhel~\cite{Bounkhel2007} proved some solution existence and uniqueness results. Namely, the author just requires that the sets $C(t)$ are prox-regular (see the definition of prox-regularity of a set below).

Since the function $x(\cdot)$ in~\eqref{sweepingprocess} can be interpreted as the trajectory of a certain mechanical system, which is driven by an external force (the gravitational force, a force generated by an electromagnetic field, a wind, etc.), several authors have studied \textit{perturbed sweeping processes} of the form 
\begin{equation}\label{ettheorem}
\begin{cases}
-\dot{x}(t)\in {\mathcal{N}}_{C(t)}(x(t)) + g(t, x(t)) \quad \text{a.e.}\ t\in [0,T]\\
x(0)=x_0,
\end{cases}
\end{equation}
where the perturbation function $g$ is either a single-valued or a multi-valued map satisfying some regularity assumptions. Since ${\mathcal{N}}_{C(t)}(x(t))=\{0\}$ if $C(t)={\mathcal H}$ for all $t\in [0,T]$, then the inclusion in~\eqref{ettheorem} reduces to the ordinary differential equation $-\dot{x}(t)=g(t, x(t))$. Hence, in that case, \eqref{ettheorem} is a Cauchy problem. In the finite-dimensional setting, where ${\mathcal H}=\mathbb R^n$, there are two celebrated theorems: the Peano theorem~\cite[Theorem~2.1, p.~10]{hartman1964} (for the solution existence of the Cauchy problem) and the Picard-Lindel\"of theorem~\cite[Theorem~1.1, p.~8]{hartman1964} (for the existence and uniqueness of the solution of the Cauchy problem). Naturally, one wishes to have some analogues of such theorems for the problem~\eqref{ettheorem}.

Perturbed sweeping processes with the sets $C(t)$, $t\in [0,T]$, being convex or the complement of the interior of a convex set were studied by Castaing et al.~\cite{CastaingDucHaValadier} and several authors in references therein. For sweeping processes with delay, where $C(t)$, $t\in [0,T]$, are assumed to be compact convex sets, Castaing and Monteiro Marques~\cite{CastaingMonteiro} obtained not only solution existence and uniqueness results but also some topological properties of the solution sets.

Bounkhel and Thibault~\cite[Corollary~3.5]{BT2005} established new characterizations of $r$-prox-regular sets in terms of the subdifferentials of the distance functions associated with the sets. Using these characterizations, they proved~\cite[Theorem~4.2]{BT2005} a solution existence theorem for nonconvex sweeping processes in Hilbert spaces with multi-valued perturbation mappings.

For discontinuous perturbed sweeping processes in the infinite-dimensional setting, Edmond and Thibault~\cite[Theorem~3.1]{ET2006} sought solutions in the form of functions of bounded variation, which can be discontinuous. As a corollary, they gave~\cite[Theorem~5.1]{ET2006} sufficient condition for  the existence of absolutely continuous solutions.

The starting point for our investigations in the present note is the papers~\cite{ET2006,ET2005} of Edmond and Thibault, where the authors investigated systematically the solution existence and uniqueness for the sweeping processes with prox-regular constraint sets $C(t)$ with single-valued perturbations. 

Based on the result on the prox-regularity of nonsmooth sublevel sets of Adly et al.~\cite[Theorem~4.1]{ANT}, we prove the solution existence as well as the solution uniqueness for a special case when $C(t)$ are sublevel sets under some assumptions. To clarify the applicability of the obtained results, we give some examples having clear mechanical interpretations. Remarkably, the examples can be solved only by invoking the uniqueness of the solution of~\eqref{ettheorem}. 

The remainder of this note is organized as follows. Section 2 presents some notions and preliminary results to be used in the sequel. In Section 3, we establish two solution existence theorems for the case where $C(t)$, $t\in[0,T]$, are sublevel sets of a certain function. In Section 4, we present four illustrative examples. 

\section{Preliminaries}

Let ${\mathcal H}$ be a Hilbert space whose scalar product will be denoted by $\langle \cdot, \cdot\rangle$ and the associated norm by $\Vert \cdot \Vert$. Denoted by $\mathbb{B}$ the closed unit ball of ${\mathcal H}$. For any set $\Omega\subset{\mathcal H}$, the notation $d(\cdot,\Omega)$ denotes the distance from a point in ${\mathcal H}$ to $\Omega$, i.e., for some $x\in {\mathcal H}$, $d(x,\Omega)= \inf\limits_{y\in \Omega}\Vert x-y\Vert.$ For any extended real number $r\in (0,\infty]$, the $r$-enlargement of $\Omega$, denoted by $U_r(\Omega)$, is defined by $U_r(\Omega)=\{x\in{\mathcal H}\mid d(x,\Omega) <r\}$. The closure, the interior and boundary of a set $\Omega\subset {\mathcal H}$ are denoted respectively by ${\rm cl}(\Omega)$, ${\rm int}(\Omega)$ and $\partial \Omega$. Given $x\in {\mathcal H}$, the set of all points $y\in \Omega$ nearest to $x$ is defined as
$$\mathbb{P}_\Omega(x) = \{y\in \Omega \mid d(x,\Omega)=\Vert x-y \Vert\}.$$ We also denote respectively by ${\mathcal N}^P_\Omega(x)$ and ${\mathcal N}_\Omega(x)$ the proximal normal cone and the Clarke normal cone of $\Omega$ at $x$, which are defined as follows.

\begin{definition} {\rm The set ${\mathcal T}_{\Omega}(x)$ is called \textit{the Clarke tangent cone} to $\Omega$ at $x$ is the set of all vector $v\in{\mathcal H}$ such that
		$$\lim\limits_{t\to 0^+, y \xrightarrow{\Omega} x} \frac{d(y+tv,\Omega)}{t}=0.$$
		\textit{The Clarke normal cone} to $\Omega$ at $x$ is the polar cone of the Clarke tangent cone, i.e., 
		$${\mathcal{N}}_\Omega(x):=\big\{ x^*\in {\mathcal H}\mid \langle x^*, v\rangle \leq 0\;\;\text {for all }\;\; v\in {\mathcal T}_\Omega(x)\big\}.$$}
\end{definition}

\begin{definition} {\rm
		A vector $v\in {\mathcal H}$ is a \textit{proximal subgradient} of a function $f:{\mathcal H}\to \mathbb{R}$ at $x$ if there exist a
		real number $\sigma \geq 0$ and a neighborhood $U$ of $x$ such that
		$$\langle v, x^\prime-x \rangle \leq f(x^\prime)-f(x)+ \sigma \Vert x^\prime-x\Vert^2,$$ 
		for all $x'\in U$.}
\end{definition}
\begin{definition} {\rm A vector $v\in {\mathcal H}$ is a \textit{proximal normal vector} to $\Omega$ at $x \in \Omega$ when it is a proximal subgradient of the indicator function of $\Omega$, that is, when there exist a constant $\sigma \geq 0$ and a neighborhood $U$ of $x$ such that $\langle v, x^\prime-x\rangle \leq  \sigma \Vert x^\prime-x \Vert^2$  for all $x^\prime\in U\cap \Omega$. The set of such vectors, which is denoted by $\mathcal{N}^P_\Omega(x)$, is said to be \textit{the proximal normal cone} of $\Omega$ at $x$.}
\end{definition}

\begin{definition}
	{\rm A nonempty closed set $\Omega$ is called \textit{$r$-prox-regular} if for all $x\in \Omega$, for all $t\in(0,r)$ and for all $\xi \in {\mathcal N}^P_{\Omega}(x)$ such that $\Vert \xi \Vert = 1$, one has $x\in \mathbb{P}_\Omega(x+t\xi)$.}
\end{definition}
A convex set is a $r$-prox-regular set for all $r>0$. Properties of prox-regular sets and their applications have been thoroughly studied in~\cite{ANT,ColomboThibault}.

	Proofs of the next proposition can be found in~the books by Benyamini and Lindenstrauss~\cite[Corollary~5.12 and Theorem~5.21]{benyamini1998} and by Diestel and Uhl~\cite[Corollary~13 of Chapter 3 and Section 6 of Chapter VII]{diestel1977}.  
		
		\begin{proposition}\label{differentiablity}
		Let $f:[a,b]\to {\mathcal H}$ be absolutely continuous. Then, $f$ is Fr\'echet differentiable almost everywhere on $[a,b]$ with respect to the Lebesgue measure.
	\end{proposition}

Let $T>0$ and $I=[0,T]$. Following Edmond and Thibault~\cite{ET2005}, we consider the next two assumptions.
\begin{assumption}{(H}{1)}
	\label{h1}
	\textit{For each $t\in I$, $C(t)$ is a nonempty closed subset of ${\mathcal H}$ which is $r$-prox-regular for some constant $r>0$.}
\end{assumption}
\begin{assumption}{(H}{2)}
	\label{h2}
	\textit{$C(t)$ varies in an absolutely continuous way, that is, there exists an absolutely continuous function $v:I\to \mathbb{R}$ such that for any $y\in {\mathcal H}$ and $s,t\in I$, one has
	\begin{equation*}\label{ET_property of C(.)}
	\Vert d(y,C(t)) - d(y,C(s)) \Vert \leq \vert v(s) - v(t)\vert.
	\end{equation*}}
\end{assumption}

The following result, which is a simplified form of Theorem 5.1 from \cite{ET2006}, provides us with an analogue of the Peano theorem~\cite[Theorem~2.1, p.~10]{hartman1964} which works for ordinary differential equations.

	\begin{theorem} {\rm (See~\cite[Theorem 5.1]{ET2006})}\label{sol_existence_0}
		Assume that a family of sets $C(t)$, $t\in I$, in ${\mathcal H}$ satisfies the assumptions~{\rm \ref{h1}} and {\rm \ref{h2}}. Assume that $G: I \times {\mathcal H}\rightrightarrows {\mathcal H}$ is a set-valued map with nonempty convex compact values such that
		\begin{enumerate}[label={\rm(\alph*)}]
			\item For any $x\in {\mathcal H}$, $G(\cdot,x)$ has a measurable selection;
			\item For all $t\in I$, $G(t,\cdot)$ is scalarly upper semicontinuous on ${\mathcal H}$;
			\item For some compact subset $K\subset \mathbb{B}$ and for some non-negative function $\beta(\cdot)\in L^1(I,\mathbb{R})$, one has for all $(t,x)\in I\times {\mathcal H}$, $$G(t,x) \subset \beta(t)(1+\Vert x \Vert)K.$$
		\end{enumerate}
		Assume also that ${\mathcal H}$ is separable if $G\not\equiv \{0\}$. Then, for any $x_0\in C(0)$, the sweeping process
		\begin{equation}\label{ettheorem_0}
		\begin{cases}
		-\dot{x}(t)\in {\mathcal{N}}_{C(t)}(x(t)) + G(t, x(t)) \quad \text{a.e.}\ t\in [0,T]\\
		x(0)=x_0,
		\end{cases}
		\end{equation}
		has at least one  absolutely continuous solution $x(\cdot)$. 
\end{theorem}

The next result is an analogue of the Picard-Lindel\"of theorem~\cite[Theorem~1.1, p.~8]{hartman1964} from the theory of ordinary differential equations.

\begin{theorem} {\rm (See~\cite[Theorem 1]{ET2005})}\label{sol_existence}
	Assume that a family of sets $C(t)$, $t\in I$, in ${\mathcal H}$ satisfies the assumptions~{\rm \ref{h1}} and {\rm \ref{h2}}. Let $g: I \times {\mathcal H}\to {\mathcal H}$ be such a separately measurable map on $I$ that
	\begin{enumerate}[label={\rm(\roman*)}]
		\item For every $\eta > 0$, there exists a non-negative function $k_\eta(\cdot) \in L^1(I,\mathbb{R})$ such that for all $t\in I$ and for any $x,y \in \overline{\mathbb{B}}(0,\eta)$ one has
		\begin{equation*}
		\Vert g(t,x)-g(t,y)\Vert \leq k_\eta(t)\Vert x - y \Vert;
		\end{equation*}
		\item There exists a non-negative function $\beta(\cdot)\in L^1(I,\mathbb{R})$ such that, for all $t\in I$ and for all $x\in \displaystyle\bigcup_{s\in I}C(s)$, one has
		$\Vert g(t,x) \Vert \leq \beta(t)(1+\Vert x \Vert).$
	\end{enumerate}
	Then, for any $x_0\in C(0)$, the sweeping process~\eqref{ettheorem} has one and only one absolutely continuous solution $x(\cdot)$. In addition, the solution satisfies the estimate
	\begin{equation*}\label{sol_bound}
	\Vert \dot{x}(t)+g(t,x(t)) \Vert\leq (1+M_{x_0})\beta(t)+\vert \dot{v}(t) \vert\ \; {\rm a.e.}\ t\in I,
	\end{equation*}
	where 
	\begin{equation*}\label{M} M_{x_0} := \Vert x_0 \Vert + \exp\left\{2\int_{0}^{T}\beta(s)ds\right\}\int_0^T\big(2\beta(s)(1+\Vert x_0 \Vert )+\vert \dot{v}(s)\vert\big)ds.\end{equation*}
\end{theorem}

 When $G$ is a single-valued mapping, Theorem~\ref{sol_existence_0} gives sufficient conditions for the existence of solution to problem~\eqref{ettheorem}. Meanwhile, Theorem~\ref{sol_existence} provides conditions for the existence and uniqueness of solution to problem~\eqref{ettheorem}. However, the assumption (c) in Theorem~\ref{sol_existence_0} is tighter than the assumption (ii) in Theorem~\ref{sol_existence}. To justify this fact, let us consider the following example.

	\begin{example}
		Let $\mathcal{H}$ be an infinite dimensional Hilbert space. Consider the problem~\eqref{ettheorem} with $C(t)$ satisfying the assumptions~\ref{h1} and~\ref{h2}. Let $g:I \times {\mathcal H}\rightrightarrows {\mathcal H}$, $g(t,x) = t\mathbb{P}_{\mathbb{B}}(x)$. We see that $g$ is linear with respect to $t$. In addition, since the projection map onto a closed convex set in Hilbert space is Lipschitz continuous, $g$ satisfies the assumptions (a), (b) of Theorem~\ref{sol_existence_0} and~(i) of Theorem~\ref{sol_existence}. Moreover, since $\Vert g(t,x) \Vert \leq t$ for all $t\in I$, the assumption~(ii) of Theorem~\ref{sol_existence} is also valid. However, the unit ball $\mathbb{B}$ in $\mathcal{H}$ is non-compact, so we cannot find any compact set $K$ such that the assumption (c) of Theorem~\ref{sol_existence_0} holds. So, it is not possible to apply Theorem~\ref{sol_existence_0} in this case. Nevertheless, for any $x_0\in C(0)$, Theorem~\ref{sol_existence} assures the solution existence and uniqueness of the problem under consideration. 
	\end{example}

\begin{remark}\label{remark_additional}{\rm Since the assumptions of the Peano theorem are weaker than those of the Picard-Lindel\"of theorem, it would be nice if one can have another version of Theorem~\ref{sol_existence_0} whose assumption set is weaker than that of Theorem~\ref{sol_existence}.}
\end{remark}

From a result of Edmond and Thibault~\cite[Proposition~2]{ET2005} it follows that, for every $t\in I$, the mapping $\psi_t:C(0)\to C(t)$ with $\psi_t(x_0):=x(x_0,t)$, where $x(x_0,\cdot)$ denotes the unique solution $x(\cdot)$ of~\eqref{ettheorem} with the initial value $x(0) =x_0$, is Lipschitz on any bounded subset of $C(0)$.

In the sequel, we will need the following characterization of a $r$-prox-regular set.

\begin{lemma}\label{prox_property1}{\rm (See~\cite[Theorem 3, p. 108]{ColomboThibault})}
	Let $\Omega$ be a closed subset of ${\mathcal H}$ and $r>0$. If $\Omega$ is $r$-prox-regular then for any $x,x^\prime \in \Omega$ and $v\in {\mathcal N}^P_\Omega(x)$, one has
	$$\langle v, x^\prime -x\rangle \leq \frac{1}{2r}\Vert v \Vert \Vert x^\prime - x\Vert^2.$$
\end{lemma}

\begin{remark}\label{proxnormal}
	{\rm If $\Omega\subset {\mathcal H}$ is $r$-prox-regular, then the proximal normal cone to $\Omega$ at any point $x\in \Omega$ coincides with the corresponding Clarke normal cone (see~\cite[Proposition 7(b)]{ColomboThibault}). So the set ${\mathcal N}^P_\Omega(x)$ in the formulation of Lemma \ref{prox_property1} can be replaced by ${\mathcal N}_\Omega(x)$.}
\end{remark}

\section{Solution Existence Theorems}
Let $T$ be a positive real number and $I=[0,T]$. Let there be given the functions $f_i: I\times {\mathcal H} \to \mathbb{R}$, $i\in \{1,\dots, m\}$. Suppose that the set
\begin{equation*}\label{ineq_system}
C(t):=\{x \in {\mathcal H} \mid f_i(t,x) \leq 0,\, i\in\{1,\dots,m\}\}
\end{equation*}
is nonempty for each $t\in I$. Assume that there is an extended real number $\rho \in [0,+\infty]$ satisfying the next four assumptions.
\begin{assumption}{(A}{1)}
	\label{a1}
	\textit{For $x\in {\mathcal H}$ and for all $i\in\{1,\dots, m \}$, $f_i(\cdot, x)$ is Lipschitz continuous with modulus $L_1>0$ on $[0,T]$.}
\end{assumption}

\begin{assumption}{(A}{2)}
	\label{a2}
	\textit{For each $t\in[0,T]$, and for all $i \in \{1,\dots, m\}$, $f_i(t,\cdot)$ is locally Lipschitz continuous on $U_\rho(C(t))$.}
\end{assumption}

\begin{assumption}{(A}{3)}
	\label{a3}
	\textit{There is $\gamma>0$ such that for all $t\in[0,T]$ and $i\in \{1,\ldots, m\}$, for all $x_1, x_2 \in U_\rho(C(t))$, and for all $\xi_j\in \partial^C f_i(t, \cdot)(x_j)$, $ j=1, 2$,
	\begin{equation*}
	\langle \xi_1-\xi_2, x_1-x_2 \rangle \geq -\gamma \Vert x_1-x_2 \Vert^2.
	\end{equation*}}
\end{assumption}

\begin{assumption}{(A}{4)}
	\label{a4}
	\textit{There is $\mu >0$ with the property that for all $t\in [0,T]$ and $x\in C(t)$ one can find $\overline{v}=v(t,x)\in {\mathcal H}$ with $\Vert\overline{v}\Vert =1$ such that for all $i\in \{1,\ldots, m\}$, for all $\xi\in \partial^C f_i(t,\cdot)(x)$, one has $\langle \xi, \overline{v}\rangle \leq -\mu.$}
\end{assumption}

Clearly, if $\partial^C f_1(t,\cdot)$ is \textit{monotone} for every $t\in [0,T]$, i.e., $\langle \xi_1-\xi_2, x_1-x_2 \rangle \geq 0$ for all $x_1,x_2\in {\mathcal H}$ and for all $\xi_j\in \partial^C f_i(t, \cdot)(x_j)$, $ j=1, 2$, then Assumption \ref{a3} is satisfied with any $\gamma >0$.

\begin{lemma}\label{Lem1} {\rm (See \cite[Theorem~4.1]{ANT})}\label{proxregularity} For all $t\in [0,T]$, the set $C(t)$ is $r$-prox-regular  with $r=\min\{\rho, \frac{\mu}{\gamma}\}$.
\end{lemma}

\begin{lemma}\label{Lem2}
	The set-valued map $C:I\rightrightarrows {\mathcal H}$ is Lipschitz with respect to the Hausdorff distance, with the Lipschitz modulus $\vartheta$, for any $\vartheta\geq \dfrac{L_1}{\mu}$.
\end{lemma}
\begin{proof}
	Fix a real number $\vartheta$ such that $\vartheta \geq \mu^{-1}L_1$. Choose a subdivision $$T_0=0<T_1<\ldots<T_p=T$$ of $[0,T]$ such that $T_k-T_{k-1}<\vartheta^{-1}\rho$ for $k=1,\dots,p$. Fix an index $k\in \{1,\dots,p\}$ and select any numbers $s,t$ from the segment $I_k:=[T_{k-1}, T_k]$. Put $u(s,t)=\vartheta \vert s-t\vert$. For any $x\in C(t)$, define $y=x+u(s,t)\overline{v}$. Since $t,s\in I_k$, we have 
	$\Vert y-x\Vert =\vartheta\vert s-t\vert <  \rho .$ This proves that $y\in {\rm int}(U_\rho(C(t)))$. By \cite[Lemma 3.2]{ANT}, for all $\lambda\in [0,1]$ we have $x+\lambda(y-x)\in {\rm int}(U_\rho(C(t))).$
	Take any $i\in \{1,\ldots, m\}$. By Assumption \ref{a2} and Lebourg's  mean value theorem (see, e.g., \cite[Theorem 2.3.7, p. 41]{Clarke}) there exists  $\lambda\in(0,1)$ such that
	\begin{equation*}
	f_i(t,y)-f_i(t,x)\in\langle \partial_2^C f_i(t,x(\lambda)), u(s,t)\overline{v}\rangle
	\end{equation*} with $x(\lambda):=(1-\lambda)x+\lambda y$.
	Hence, by Assumptions \ref{a1} and \ref{a4} we have
	\begin{equation*}
	\begin{aligned}
	f_i(s,y)&=[f_i(s, y)-f_i(t, y)]+f_i(t,x)+[f_i(t,y)-f_i(t,x)]\\
	&\leq L_1\vert s-t\vert -u(s,t)\mu\\
	&=\left(L_1-\vartheta\mu \right)\vert s-t\vert.
	\end{aligned}
	\end{equation*} Hence, $f_i(s,y)\leq 0$. Since $i\in\{1,\ldots, m\}$ can be chosen arbitrarily, we have thus shown that the vector $y=x+\vartheta\vert s-t\vert \overline{v}$ belongs to $C(s)$. So, $d(x,C(s)) \leq \vartheta\vert s-t \vert$ for every $x\in C(t)$. By symmetry, we get $d(x^\prime,C(t))\leq \vartheta\vert s-t \vert$ for every $x^\prime\in C(s)$. Consequently, we obtain $d_H(C(t),C(s)) \leq \vartheta\vert t - s \vert.$
	
	The proof is complete. \end{proof}

	\begin{theorem}\label{theorem_main0}
		Suppose that Assumptions {\rm \ref{a1}--\ref{a4}} are fulfilled. Let $g: I \times {\mathcal H}\to {\mathcal H}$ satisfy the three requirements {\rm (a)}, {\rm (b)} and {\rm (c)} in Theorem~\ref{sol_existence_0}. Then, for any $x_0\in C(0)$, the sweeping process
		\begin{equation}\label{mainproblem}
		\begin{cases}
		-\dot{x}(t)\in {\mathcal{N}}_{C(t)}(x(t)) + g(t, x(t)) \quad \text{a.e.}\ t\in I\\
		x(0)=x_0
		\end{cases}
		\end{equation}
		has at least one absolutely continuous solution $x(\cdot)$.
\end{theorem}

\begin{proof} By Lemma \ref{Lem1}, the set $C(t)$ is $r$-prox-regular for all $t\in[0,T]$. Moreover, Lemma \ref{Lem2} states that 
	$$d_H(C(t),C(s)) \leq \vartheta\vert t - s \vert.$$
	For all $y\in{\mathcal H}$, we have that $\Vert d(y,C(t)) - d(y,C(s)) \Vert \leq d_H(C(t),C(s))$. It follows that $C(t)$ varies in an absolutely continuous way, i.e.,
	$$\Vert d(y,C(t)) - d(y,C(s)) \Vert \leq \vert v(s) - v(t)\vert,$$
	where $v: I\to \mathbb{R}$, $v(z) = \vartheta z$. By Theorem~\ref{sol_existence_0}, we obtain the desired result.
\end{proof}

\begin{theorem}\label{theorem_main1}
	Suppose that Assumptions {\rm \ref{a1}--\ref{a4}} are fulfilled. Let $g: I \times {\mathcal H}\to {\mathcal H}$ be such a separately measurable map on $I$ that satisfies the two requirements~{\rm (i)} and {\rm (ii)} in Theorem~\ref{sol_existence}. Then, for any $x_0\in C(0)$, the sweeping process \eqref{mainproblem} has a unique absolutely continuous solution $x(\cdot)$.
\end{theorem}

\begin{proof} Using Theorem~\ref{sol_existence} instead of Theorem~\ref{sol_existence_0} and arguing similarly as in the proof of Theorem~\ref{theorem_main0}, one can obtain the desired result.	
\end{proof}

\begin{remark}\label{independence_on_x0}
	{\rm The assumptions \ref{a1}--\ref{a4} on the functions $f_i$, $i\in\{1,\dots,m\}$, and the family of sets $C(t)$, $t\in I$, do not depend on the choice of $x_0$ from $C(t)$. Clearly, the requirements {\rm (i)} and {\rm (ii)} on $g(t,x)$ in the formulation of Theorem~\ref{theorem_main1} also do not depend on the choice of $x_0$ from $C(t)$. }
\end{remark}

\section{Applications to Mechanical Sweeping Processes}
To illustrate the applicability of Theorem~\ref{theorem_main1}, we shall provide examples considering classical mechanical models.

\begin{example}\label{ex1}\rm
	Consider the problem~\eqref{mainproblem} with ${\mathcal H}=\mathbb{R}^2$, $m=1$, $f_1(t,x) =t-x_2+\vert x_1\vert$, and $g(t,x) = 0$ for all $t\in [0,T]$, $x=(x_1,x_2)\in\mathbb{R}^2$. Here, we have 
	\begin{equation}\label{levelset}
	C(t) = \{x\in \mathbb{R}^2\mid -x_2+\vert x_1\vert \leq -t\}.
	\end{equation} Let the initial condition be $x(0) = (0,0)$. Obviously, $x(0) \in C(0)$ and $f=f_1$ satisfies Assumptions~\ref{a1} and \ref{a2}. We have
	\begin{equation}\label{subdifferential_ex1}
	\partial^Cf_1(t,\cdot)(x)=\begin{cases}
	\{(1,-1)\} \qquad \quad&\text{if} \; x_1 >0\\
	[-1,1]\times \{-1\}&\text{if}\; x_1=0\\
	\{(-1,-1)\} &\text{if} \; x_1 <0.
	\end{cases}
	\end{equation}
	Since $f_1(t,\cdot)$ is convex, $\partial^C f_1(t,\cdot)$ coincides with the convex subdifferential mapping of $\partial f_1(t,\cdot)$, which is monotone. Hence, for any $t\in[0,T]$, the mapping $\partial^C f_1(t,\cdot)$ is hypermonotone with any $\gamma>0$. Thus, Assumption \ref{a3} is satisfied. Now, to check Assumption~\ref{a4}, let us fix any $\mu \in (0,1]$. Suppose that $t\in[0,T]$ and $x\in C(t)$ are given arbitrarily. For $\overline{v}:=(0,1)$, one has $\langle \xi,\overline{v} \rangle=\xi_2$, where $\xi=(\xi_1,\xi_2) \in \partial^Cf_1(t,\cdot)(x)$ can be chosen arbitrarily. Thanks to~\eqref{subdifferential_ex1}, we have $\xi_2=-1$. Hence, $$\langle \xi,\overline{v} \rangle=-1\leq -\mu.$$ 
	We have thus showed that Assumption~\ref{a4} is satisfied. Since $g(t,x)\equiv 0$, the requirements (i) and (ii) on $g$ are fulfilled. So, according to Theorem \ref{theorem_main1}, \eqref{mainproblem} has a unique absolutely continuous solution $x(\cdot)$. Interestingly, we can give an explicit formula for~$x(\cdot)$. Namely, let us show that
	\begin{equation}\label{examplesol}
	x_1(t) = 0, \; x_2(t)=t \quad \forall t\in [0,T].
	\end{equation}
	Clearly, the trajectory $x(t)$ given by \eqref{examplesol} satisfies the conditions $$x(0) = (0,0)\quad  {\rm and}\quad  -\dot{x}(t) = (0,-1).$$ Since $C(t)$ is convex, the Clarke normal cone to $C(t)$ at any point of $C(t)$ coincides with the normal cone to $C(t)$ at that point in the sense of convex analysis (see \cite[Proposition 2.4.4]{Clarke}). So, applying~\cite[Proposition 2, p. 206]{IT} to the set $C(t)$ in~\eqref{levelset}, which is a sublevel set of the continuous convex function $f_1(t,\cdot)$, at the boundary $x(t)=(x_1(t), x_2(t))$, one obtains ${\mathcal N}_{C(t)}(x(t))= \mathbb{R}_+\partial f_1(t,\cdot)(x(t)).$ Since $x_1(t) \equiv 0$, combining this with \eqref{subdifferential_ex1} gives $${\mathcal N}_{C(t)}(x(t))= \mathbb{R}_+([-1,1]\times \{-1\}).$$ So, $-\dot{x}(t) \in {\mathcal N}_{C(t)}(x(t))+g(t,x(t))$ for all $t \in [0,T]$. Hence, formula \eqref{examplesol} describes the unique absolutely continuous solution of the problem in question. The above mathematical model and the solution have the following clear mechanical meanings. In the horizontal coordinate plane $\mathbb{R}^2$, there is a small metal ball standing at the origin of the plane at time $t=0$. The boundary of $C(0)$ is the union of two orthogonal half-lines. Suppose that the boundary is the frame made from two long sticks of bamboo or wood which are firm enough that they cannot be bend by the metal ball. The set $C(t)$ in~\eqref{levelset} is the position of $C(0)$ at the time $t$. The requirement saying that the ball must be inside $C(t)$ at any time $t$ means that it must be in the plane area formed by the frame. The change of $C(t)$ with respect to $t$ corresponds to the movement of the frame along the $x_2$-axis with the velocity $1$. The assumption $g(t,x) \equiv 0$ means that there is no external force acting on the ball. The formula~\eqref{examplesol} of the obtained solution means that the ball always lies in the corner of the frame, when the later moves steadily along the $x_2$-axis.	
\end{example}

Concerning the sweeping problem in Example~\ref{ex1}, we observe that the role of the normal cone operator ${\mathcal{N}}_{C(t)}(x(t))$ in the inclusion $-\dot{x}(t)\in {\mathcal{N}}_{C(t)}(x(t)) + g(t, x(t))$ is important. Namely, note that the last inclusion implies $x(t) \in C(t)$. Note also that $0\in {\mathcal{N}}_{C(t)}(x(t))$ if $x(t)\in C(t)$. So, together with~\eqref{mainproblem}, it is naturally to consider the following tighter problem:
\begin{equation*}
\begin{cases}
-\dot{x}(t) = g(t, x(t))  \quad \text{a.e.} \ t\in I\\
x(t)\in C(t)\quad \text{for} \ t\in I\\
x(0)=0
\end{cases}
\end{equation*}
Since $g(t,x)\equiv 0$, the first and the third conditions of this system imply that $x(t) = 0$ for all $t\in I$. However, for this curve $x(t)$, the second condition of the system is violated. So, \textit{the assertion of Theorem~\ref{theorem_main1} may fail to hold if one replaces the inclusion $-\dot{x}(t)\in {\mathcal{N}}_{C(t)}(x(t)) + g(t, x(t))$ by the conditions $-\dot{x}(t) = g(t, x(t))$ and $x(t)\in C(t)$.}

\begin{example}\label{ex2}\rm
	Consider problem \eqref{mainproblem} with the data given in Example~\ref{ex1}, where the initial point is $x(0) = x_0$ with $x_0=(x_1^0,x_2^0)$ being an arbitrary point from $C(0)$. The analysis in Example~\ref{ex1} shows that the assumptions \ref{a1}--\ref{a4} and the requirements~(i) and (ii) on $g(t,x)$ in the formulation of Theorem~\ref{theorem_main1} are satisfied. Hence, by Remark~\ref{independence_on_x0} and Theorem~\ref{theorem_main1}, the sweeping process~\eqref{mainproblem} has a unique absolutely continuous solution $x(\cdot)$. To have an explicit formula for this solution $x(\cdot)$, we first suppose that $x_0$ belongs to the interior of $C(0)$. This means that $\vert x^0_1\vert < x^0_2$. Put $\overline{t}_{x_0} = x^0_2-\vert x^0_1\vert$ and note that $\overline{t}_{x_0} > 0$.
	
	\textit{Case 1}: $T \leq \overline{t}_{x_0}$. In this case, since $f_1(t,x_0) = t-x_2^0+\vert x_1^0 \vert= t- \overline{t}_{x_0}<0$ for all $t\in [0,T)$, one has $x_0 \in {\rm int}(C(t))$ for all $t\in [0,T)$. So, setting $x(t) = x_0$ for $t\in I$, we obtain $${\mathcal N}_{C(t)}(x(t)) = \{(0,0)\}$$ for all $t\in [0,T).$ Therefore, \eqref{mainproblem} is satisfied. Since the solution is unique by Theorem~\ref{theorem_main1}, the just defined constant trajectory is the unique absolutely continuous solution of the sweeping process under our consideration.
	
	\textit{Case 2}: $\overline{t}_{x_0}<T$. First, consider the subcase where $\overline{t}_{x_0} \leq 2\vert x_1^0\vert+\overline{t}_{x_0} < T$. Let us prove that the unique solution $x(\cdot)$ can be given by the formula
	\begin{equation}\label{case2_ex2}
	x(t)=\begin{cases}
	x_0\qquad \quad&\text{if} \; t\in[0,\overline{t}_{x_0})\\
	(x_1^0-{\rm sign}(x_1^0)\frac{t-\overline{t}_{x_0}}{2},x_2^0+\frac{t-\overline{t}_{x_0}}{2})&\text{if}\; t\in [\overline{t}_{x_0},2\vert x_1^0\vert+\overline{t}_{x_0})\\
	(0,t) &\text{if} \; t\in[2\vert x_1^0\vert+\overline{t}_{x_0}, T].
	\end{cases}
	\end{equation}
	Note that the function $x(\cdot)$ is absolutely continuous on $[0,T]$ and $x(0)=x_0$. Arguing as in Case~1, we obtain $-\dot{x}(t)\in {\mathcal{N}}_{C(t)}(x(t)) + g(t, x(t))$ for every $t\in [0,\overline{t}_{x_0})$. For $t\in [\overline{t}_{x_0},2\vert x_1^0\vert +\overline{t}_{x_0})$, if $x_1^0 \leq 0$ then $ t < -2x_1^0+\overline{t}_{x_0}$. Hence, $x_1(t)=x_1^0+\frac{t-\overline{t}_{x_0}}{2} <0.$ Combining this with~\eqref{case2_ex2} yields
	$$f_1(t,x(t))=t-x_2(t)+\vert x_1(t)\vert =t-\left(x_2^0+\frac{t-\overline{t}_{x_0}}{2}\right)-\left(x_1^0+\frac{t-\overline{t}_{x_0}}{2}\right)=0.$$
	This means that $x(t)\in \partial C(t)$, $x_1(t) < 0$; so $\partial^Cf_1(t,x(t)) = \{(-1,-1)\}$. Thanks to the continuity and convexity of the function $f_1(t,\cdot)$, applying~\cite[Proposition 2.4.4]{Clarke} we have ${\mathcal N}_{C(t)}(x(t))=\mathbb{R}_+ \{(-1,-1)\}.$ It follows that $\dot{x}(t) = (\frac{1}{2},\frac{1}{2})\in -{\mathcal N}_{C(t)}(x(t))$ for all $t\in (\overline{t}_{x_0},-2x_1^0+\overline{t}_{x_0})$. The situation $x_1^0 > 0$ can be treated similarly. Therefore, $-\dot{x}(t)\in {\mathcal{N}}_{C(t)}(x(t)) + g(t, x(t))$ for $t\in (\overline{t}_{x_0},2\vert x_1^0\vert+\overline{t}_{x_0})$. Now, for $t\in[2\vert x_1^0 \vert+\overline{t}_{x_0}, T]$, one has $-\dot{x}(t)\in {\mathcal{N}}_{C(t)}(x(t)) + g(t, x(t))$ by~\eqref{case2_ex2} and the result given in Example~\ref{ex1}.  Therefore, \eqref{case2_ex2} describes the unique absolutely continuous solution $x(\cdot)$ of the problem in question. In the situation where $\overline{t}_{x_0} < T\leq 2\vert x_1^0\vert+\overline{t}_{x_0}$, arguing analogously as before, we can show that the formula 
	\begin{equation}\label{case3_ex2}
	x(t)=\begin{cases}
	x_0\qquad \quad&\text{if} \; t\in[0,\overline{t}_{x_0})\\
	(x_1^0-{\rm sign}(x_1^0)\frac{t-\overline{t}_{x_0}}{2},x_2^0+\frac{t-\overline{t}_{x_0}}{2})&\text{if}\; t\in [\overline{t}_{x_0},T]
	\end{cases}
	\end{equation}
	describes the unique absolutely continuous solution $x(\cdot)$ of our problem.
	
	Now, suppose that $x(0) \in \partial C(0)$. This means that $x_2^0-\vert x_1^0\vert =0$. This situation reduces to Case 2 above with $\overline{t}_{x_0}:=x_2^0-\vert x_1^0\vert=0$. So, the unique absolutely continuous solution $x(\cdot)$ of our problem is given by
	\begin{equation}\label{case4_ex2}
	x(t)=\begin{cases}
	(x_1^0-{\rm sign}(x_1^0)\frac{t}{2},x_2^0+\frac{t}{2})&\text{if}\; t\in [0,2\vert x_1^0\vert)\\
	(0,t) &\text{if} \; t\in[2\vert x_1^0\vert, T]
	\end{cases}
	\end{equation}
	whenever $2\vert x_1^0\vert<T$, and
	\begin{equation}\label{case5_ex2}
	x(t)=\left(x_1^0-{\rm sign}(x_1^0)\frac{t}{2},x_2^0+\frac{t}{2}\right) \ \; \text{for} \;t\in[0,T]
	\end{equation}
	whenever $2\vert x_1^0\vert\geq T$. As in the preceding example, the problem here and the obtained solution can be interpreted respectively as a mechanical problem and a mechanical motion as follows. Suppose that,  at time $t=0$, there is a small metal ball standing at the point $x_0\in C(0)$ in the horizontal plane $\mathbb{R}^2$. When the set $C(0)$ moves along the $x_2$-axis with the velocity $1$ (see~\eqref{levelset}), its boundary - a firm frame consisting of two orthogonal half-lines - also moves along the $x_2$-axis with the velocity $1$. The ball cannot overpass the frame. If $x_0\in{\rm int}(C(0))$, $x_1(0)\neq 0$, and $2\vert x_1^0\vert+\overline{t}_{x_0} < T$ with $\overline{t}_{x_0}:=x_2^0-\vert x_1^0\vert$, then~\eqref{case2_ex2} shows that the motion of the ball in the time segment $[0,T]$ and has three phases: (a)~Until the time instant $\overline{t}_{x_0}$, the ball stays still; (b)~In the time interval $[\overline{t}_{x_0},2\vert x_1^0\vert+\overline{t}_{x_0})$, the ball goes steadily along one wing of the boundary of $C(0)$ with the speed $\frac{\sqrt{2}}{2}$ (the ball is on the left wing if $x_1(0)< 0$ and it is on the right wing if $x_1(0)>0$); (c)~In the time interval $[2\vert x_1^0\vert+\overline{t}_{x_0}, T]$, the ball always lies in the corner of the above-mentioned frame. Similar interpretations can be given for formulas~\eqref{case3_ex2}--\eqref{case5_ex2}.
\end{example}

Let the horizontal plane $\mathbb{R}^2$ in the preceding example be replaced by a \textit{vertical} plane $\mathbb{R}^2$, where the $x_2$-axis is orthogonal to the earth surface and pointing up. Then, the set $C(t)$ given by~\eqref{levelset} can be interpreted as the position of the set $$C(0) = \{x\in \mathbb{R}^2\mid -x_2+\vert x_1\vert \leq 0\}$$ at time $t$. In other words, in accordance with formula~\eqref{levelset}, the set $C(0)$ is moving up along the $x_2$-axis with the velocity $1$. As before, the boundary of $C(0)$ - a firm frame - also moves along the $x_2$-axis with the velocity $1$. Note that the metal ball in question cannot overpass the frame. Since the ball has the tendency to go down  straightly with the acceleration $g_0=9.8$, the velocity of its free fall is $-g_0t$. So the equation of motion of the ball should be $-\dot{x}(t)\in {\mathcal{N}}_{C(t)}(x(t)) + g(t, x(t))$ for almost everywhere $t\in I$, where $g(t,x):=(0,g_0t)$. The solution of this mechanical problem is given below. 

\begin{example}\label{ex3}\rm
	Consider problem \eqref{mainproblem} with the data given in Example~\ref{ex1} except for $g(t,x)=(0,g_0t)$, where $g_0=9.8$ is the gravitational acceleration. Let the initial condition be $x(0) = (x_1^0,x_2^0)$. As we know from the two examples above, the assumption~\ref{a1}--\ref{a4} hold true. Since $g(t,x)$ is independent of the second variable, it is clear that the requirement (i) in Theorem~\ref{theorem_main1} is satisfied. In addition, as $g(t,x)$ is a linear function of $t$, the requirement (ii) in the theorem is satisfied with the choice $\beta(t) =g_0t$. Hence, by Remark~\ref{independence_on_x0} and Theorem~\ref{sol_existence}, the sweeping process~\eqref{mainproblem} has a unique absolutely continuous solution $x(\cdot)$. To provide an explicit formula for this solution $x(\cdot)$, we first consider the situation where $x_0\in{\rm int}(C(0))$. Putting $\overline{t}_{x_0} = x^0_2-\vert x^0_1\vert$, one has $\overline{t}_{x_0} > 0$. Define $\theta_{x_0}^1=\frac{-1+\sqrt{1+2g_0\overline{t}_{x_0}}}{g_0}$ and $\theta_{x_0}^2=\frac{-1+\sqrt{1+2g_0(\overline{t}_{x_0}+2\vert x_1^0\vert)}}{g_0}$. It is clear that $0<\theta_{x_0}^1\leq\theta_{x_0}^2$.
	
	\textit{Case 1}: $T\leq \theta_{x_0}^1$. Setting 
	\begin{equation}\label{case1_ex3}
	x(t)=(x_1^0,x_2^0-\frac{g_0t^2}{2}) \quad (\forall t\in I),
	\end{equation}
	we have $f_1(t,x(t)) = t- x_2^0+\frac{g_0t^2}{2}+\vert x_1^0 \vert<0,$ for any $t\in [0,T)$. Hence, $x(t) \in {\rm int}(C(t))$ for all $t\in [0,T)$. So, ${\mathcal N}_{C(t)}(x(t))=\{(0,0)\}$ for all $t\in [0,T)$. Since $-\dot{x}(t)= (0,g_0t)$, it follows that the inclusion in~\eqref{mainproblem} is satisfied for all $t\in [0,T)$. Therefore, Theorem~\ref{theorem_main1} assures that the chosen trajectory is the unique absolutely continuous solution of~\eqref{mainproblem}.
	
	\textit{Case 2}: $\theta_{x_0}^1 < T$. If $\theta_{x_0}^1 \leq \theta_{x_0}^2 < T.$ then the explicit formula for the solution $x(\cdot)$ is
	
	\begin{equation}\label{case2_ex3}
	x(t)=\begin{cases}
	(x_1^0,x_2^0-\frac{g_0t^2}{2})\qquad \quad&\text{if} \; t\in[0,\theta_{x_0}^1)\\
	(x_1^0-{\rm sign}(x_1^0)\left(\frac{t-\overline{t}_{x_0}}{2}+\frac{g_0t^2}{4}\right),x_2^0+\frac{t-\overline{t}_{x_0}}{2}-\frac{g_0t^2}{4})&\text{if}\; t\in [\theta_{x_0}^1,\theta_{x_0}^2)\\
	(0,t) &\text{if} \; t\in[\theta_{x_0}^2, T].
	\end{cases}
	\end{equation}
	Indeed, the function $x(\cdot)$ is an absolutely continuous on $[0,T]$, $x(0)=x_0$, and a direct verification shows that $-\dot{x}(t) \in {\mathcal N}_{C(t)}(x(t)) + (0,g_0t)$ for $t\in [0,\theta_{x_0}^1)$. Now, suppose that $x_1^0 \leq 0$. Then we have $x_1(t)=x_1^0+\frac{t-\overline{t}_{x_0}}{2}+\frac{g_0t^2}{4}<0$ for $t\in[\theta_{x_0}^1,\theta_{x_0}^2)$.  So, for $t\in[\theta_{x_0}^1,\theta_{x_0}^2)$, one has
	$$f_1(t,x(t))= t-x_2^0-\frac{t-\overline{t}_{x_0}}{2}+\frac{g_0t^2}{4}-x_1^0-\frac{t-\overline{t}_{x_0}}{2}-\frac{g_0t^2}{4}=0.$$
	Hence, $x(t) \in \partial C(t)$. Since $x_1(t)<0$, this implies that $\partial^Cf_1(t,x(t)) =\{-1,-1\}$ for every $t\in[\theta_{x_0}^1,\theta_{x_0}^2)$. Thanks to the continuity and convexity of $f_1(t,\cdot)$, applying~\cite[Proposition 2.4.4]{Clarke}, we obtain  ${\mathcal N}_{C(t)}(x(t))=\mathbb{R}_+ \{(-1,-1)\}$ for $t\in[\theta_{x_0}^1,\theta_{x_0}^2)$. Since 
	$$\dot{x}(t) = \left(\frac{1+g_0t}{2},\frac{1+g_0t}{2}-g_0t\right),$$
	one has $-\dot{x}(t) \in {\mathcal N}_{C(t)}(x(t)) +(0,g_0t)$ for $t\in (\theta_{x_0}^1,\theta_{x_0}^2)$. Thus, for every $t\in (\theta_{x_0}^1,\theta_{x_0}^2)$, the inclusion $-\dot{x}(t) \in {\mathcal N}_{C(t)}(x(t)) +g(t,x)$ holds. For $t\in(\theta_{x_0}^2, T)$, it is clear that $f_1(t,x(t))= 0$
	and $\dot{x}(t)=(0,1)$. Since ${\mathcal N}_{C(t)}(x(t))= \mathbb{R}_+([-1,1]\times \{-1\})$, the inclusion $$-\dot{x}(t) \in {\mathcal N}_{C(t)}(x(t))+ g(t,x)$$ holds for $t\in(\theta_{x_0}^2, T)$. Therefore, the function $x(\cdot)$ given in~\eqref{case2_ex3} describes the unique absolutely continuous solution of the problem under consideration. The situation $x_1^0 > 0$ can be treated similarly. If $T \leq \theta_{x_0}^2$, arguing analogously, we can prove that the formula 
	\begin{equation}\label{case3_ex3}
	x(t)=\begin{cases}
	(x_1^0,x_2^0-\frac{g_0t^2}{2})\qquad \quad&\text{if} \; t\in[0,\theta_{x_0}^1)\\
	(x_1^0-{\rm sign}(x_1^0)\left(\frac{t-\overline{t}_{x_0}}{2}+\frac{g_0t^2}{4}\right),x_2^0+\frac{t-\overline{t}_{x_0}}{2}-\frac{g_0t^2}{4})&\text{if}\; t\in [\theta_{x_0}^1,T]
	\end{cases}
	\end{equation}
	describes the unique solution $x(\cdot)$.
	
	Now, suppose that $x(0)\in \partial C(0)$. It is not difficult to show that the unique absolute solution $x(\cdot)$ is described as
	\begin{equation}\label{case4_ex3}
	x(t)=\begin{cases}
	(x_1^0-{\rm sign}(x_1^0)\left(\frac{t}{2}+\frac{g_0t^2}{4}\right),x_2^0+\frac{t}{2}-\frac{g_0t^2}{4})&\text{if}\; t\in [0,\theta_{x_0}^2),\\
	(0,t) &\text{if} \; t\in[\theta_{x_0}^2, T].
	\end{cases}
	\end{equation}
	if $\theta_{x_0}^2 < T$, and by the formula
	\begin{equation}\label{case5_ex3}
	x(t)= \left(x_1^0-{\rm sign}(x_1^0)\left(\frac{t}{2}+\frac{g_0t^2}{4}\right),x_2^0+\frac{t}{2}-\frac{g_0t^2}{4}\right)\quad\text{for}\; t\in [0,T].
	\end{equation}
	if $\theta_{x_0}^2 \geq T$. The mechanical meanings of the motion modes~\eqref{case1_ex3}--\eqref{case5_ex3}  of the metal ball are similar to those explained in Example~\ref{ex2}.  
\end{example}

\begin{remark}\label{reachability}
	{\rm By ${\mathcal R}_T$ we denote the set of end points of the sweeping process~\eqref{ettheorem}, i.e., the set of all $x(T)$ with $x(\cdot)$ being the unique solution of~\eqref{mainproblem} where $x_0\in C(0)$ is chosen arbitrarily. It is an interesting question that under which conditions on $C(t)$, $t\in[0,T]$, we have ${\mathcal R}_T=C(T)$. The following example shows that even when $C(t)$ is just a linear translation of $C(0)$, we get a negative answer. The system 
		\begin{equation}\label{reverseproblem}
		\begin{cases}
		-\dot{x}(t)\in {\mathcal{N}}_{C(t)}(x(t)) + g(t, x(t)) \quad \text{a.e.}\ t\in I,\\
		x(T)=x_1
		\end{cases}
		\end{equation}
		will be used in our analysis.}
\end{remark}
\begin{example}\rm
	Consider problem \eqref{mainproblem} with ${\mathcal H}=\mathbb{R}^2$, $m=2$, $$f_1(t,x) =t-x_2+\vert x_1\vert,\ \;f_2(t,x)= x_2-t-1,$$ and $g(t,x) = 0$ for all $t\in [0,T]$, where $T=3$, and $x=(x_1,x_2)\in\mathbb{R}^2$. Here, we have 
	\begin{equation}\label{ex4_levelset}
	C(t) = \{x\in \mathbb{R}^2\mid -x_2+\vert x_1\vert \leq -t,\ x_2\leq t+1\}.
	\end{equation} Let the terminal condition be $x(T) = (x_1^1,x_2^1)$. If~\eqref{reverseproblem} has a solution $x(\cdot)$, then one has $x(0)=x_0$ for some $x_0=(x_1^0,x_2^0)\in C(0)$. Since the assumptions \ref{a1}--\ref{a4} and the requirements~(i) and (ii) on $g(t,x)$ in the formulation of Theorem~\ref{theorem_main1} are satisfied, by Remark~\ref{independence_on_x0} and Theorem~\ref{theorem_main1}, the sweeping process~\eqref{mainproblem} with the chosen $x_0$ has a unique absolutely continuous solution. Using the formula of $C(t)$ in~\eqref{ex4_levelset}, one can easily show that $\vert x_1^0 \vert \leq 1$. For $\overline{t}_{x_0}:=x_2^0-\vert x_1^0\vert$, we have $2\vert x_1^0\vert+\overline{t}_{x_0} \leq 2<T$.  Arguing similarly to Example~\ref{ex2}, we can show that the unique absolutely continuous solution $x(\cdot)$ of~\eqref{mainproblem} is given by~\eqref{case2_ex2} if $x_2^0 = 1$ or if $x_0\in {\rm int}(C(0))$, and by~\eqref{case4_ex2} if $x_2^0 < 1$ and $x_0\in\partial C(0)$. In both cases, we have $x(T) = (0,3)$. So, the following assertions are valid: (i) If $x_1\neq (0,3)$, then problem~\eqref{reverseproblem} has no solution; (ii) If $x_1=(0,3)$, then~\eqref{reverseproblem} have infinite number of solutions; (iii) For any $x_0\in C(0)$, the unique solution $x(\cdot)$ of~\eqref{ettheorem} ends at the point $x(T)=(0,3)$.
\end{example}

However, if the sweeping process~\eqref{reverseproblem} in Remark~\ref{reachability} is subjected to multi-valued perturbations $g(t,x(t))$, then the above question can be considered as a controllability problem, for which we expect to have a positive solution. This issue will be addressed in our further work.

\section{Conclusions}
In this paper, the solution existence as well as the solution uniqueness for perturbed sweeping processes has been studied under the assumption of the prox-regularity of the constraint sets. 

If the perturbation function $g(t,x)$ is multi-valued, then we have deal with multi-valued perturbed sweeping processes in the prox-regular case. For these problems, it is of interest to establish some results on the solution existence, continuous dependence of the solutions, and the reachability of sweeping processes similar to the ones given in the present paper.

It is unclear to us whether one can relax the assumptions of Theorem~\ref{theorem_main0} in such a way that the solution existence of the problem~\eqref{mainproblem} is still guaranteed, or not.

\vskip 6mm
\noindent{\bf Acknowledgments}

\noindent  The authors are grateful to Professor Samir Adly for stimulating discussions on the subject. Nguyen Nang Thieu and Nguyen Dong Yen were supported by the project \textit{"Some qualitative properties of optimization problems and dynamical systems, and applications"} (Code: ICRTM01$\_$2020.08)  of the International Center for Research and Postgraduate Training in  Mathematics (ICRTM) under the auspices of UNESCO of Institute of Mathematics, Vietnam Academy of Science and Technology.


\begin{thebibliography}{99}
	
\bibitem{Clarke} F.H. Clarke, Optimization and Nonsmooth Analysis, John Wiley \& Sons, Inc., New York, 1983. 

\bibitem{IT} A.D. Ioffe and V.M. Tihomirov, Theory of Extremal Problems, North-Holland Publishing Co., Amsterdam-New York, 1979.

\bibitem{Moreau4} J.-J. Moreau, Rafle par un convexe variable. I, In: Travaux du S\'eminaire
d’Analyse Convexe, Vol. I, Exp. No. 15, 43 pages, Secr\'etariat des Math., Publ. No. 118, 1971.

\bibitem{Moreau5} J.-J. Moreau, Rafle par un convexe variable. II, In: Travaux du S\'eminaire
d’Analyse Convexe, Vol. II, Exp. No. 3, 36 pages, Secr\'etariat des Math.,
Publ. No. 122, 1972.

\bibitem{ColomboGoncharov} G. Colombo, V.V. Goncharov, The sweeping processes without convexity,
Set-Valued Anal., 7 (1999), 357-374.

\bibitem{Bounkhel2007} M. Bounkhel, Existence and uniqueness of some variants of nonconvex sweeping processes, J. Nonlinear Convex Anal., 8 (2007), 311-323.

\bibitem{hartman1964} P. Hartman, Ordinary Differential Equations, John Wiley \& Sons, New York, 1964.

\bibitem{CastaingDucHaValadier} C. Castaing, T.X. Duc Ha, M. Valadier, Evolution equations governed by
the sweeping process, Set-Valued Anal., 1 (1993), 109-139.

\bibitem{CastaingMonteiro} C. Castaing, M.D.P. Monteiro Marques, Topological properties of solution
sets for sweeping processes with delay, Portugal. Math., 54 (1997), 485-507.

\bibitem{BT2005} M. Bounkhel, L. Thibault, Nonconvex sweeping process and prox-regularity
in Hilbert space, J. Nonlinear Convex Anal., 6 (2005), 359-374.

\bibitem{ET2006} J.F. Edmond, L. Thibault, BV solutions of nonconvex sweeping process
differential inclusion with perturbation, J. Differential Equations 226 (2006), 135-179.

\bibitem{ET2005} J.F. Edmond, L. Thibault, Relaxation of an optimal control problem involving a perturbed sweeping process, Math. Program., 104 (2005), Ser. B, 347-373.
	
\bibitem{ANT} S. Adly, F. Nacry, L. Thibault, Preservation of prox-regularity of sets with applications to constrained optimization, SIAM J. Optim., 26 (2016), 448-473.

\bibitem{ColomboThibault} G. Colombo, L. Thibault, Prox-regular sets and applications, In: Handbook of
Nonconvex Analysis and Applications, pp. 99-182, Int. Press, Somerville, MA, 2010.

\bibitem{benyamini1998} Y. Benyamini, J. Lindenstrauss, Geometric Nonlinear Functional Analysis, American Mathematical Society, 1998.

\bibitem{diestel1977} J. Diestel, J.J. Uhl, Jr., Vector Measures, American Mathematical Society, Providence, R.I., 1977. 

\end{thebibliography}
\end{document}